%% file: weightedtest.tex
\documentclass[preprint,authoryear,12pt]{elsarticle}

 \usepackage{pgfplots}
\usepackage{amssymb}
 \usepackage{amsthm, amsmath}
 
 \newtheorem{theorem}{Theorem}[section]
\newtheorem{corollary}[theorem]{Corollary}

\newtheorem{proposition}[theorem]{Proposition}

\begin{document}
\title{Non-Parametric Weighted Tests for Independence Based on Empirical Copula Process}
\author{Ivan Medovikov \\\vspace{6mm} \em{Department of Economics, Brock University \\ 500 Glenridge Ave., St. Catharines, L2S 3A1, Canada \\ Tel.: +1 905 688 55 50 ext. 6148, Email: imedovikov@brocku.ca}}

\begin{abstract}
We propose a class of flexible non-parametric tests for the presence of dependence between components of a random vector based on weighted Cram\'{e}r-von Mises functionals of the empirical copula process. The weights act as a tuning parameter and are shown to significantly influence the power of the test, making it more sensitive to different types of dependence. Asymptotic properties of the test are stated in the general case, for an arbitrary bounded and integrable weighting function, and computational formulas for a number of weighted statistics are provided. Several issues relating to the choice of the weights are discussed, and a simulation study is conducted to investigate the power of the test under a variety of  dependence alternatives. The greatest gain in power is found to occur when weights are set proportional to true deviations from independence copula.

%\begin{keywords}
%Independence tests; non-parametric methods; empirical copula process; weighted Cram\'{e}r-von Mises statistics; monte-carlo; power study
%\end{keywords}

%\begin{classcode}
%62G10; 62H15; 62H10; 62G09; 65C05
%\end{classcode}

\end{abstract}

\maketitle

\section{Introduction}
\label{introduction}
We propose a class of flexible non-parametric tests for detecting dependence between $d \geq 2$ scalar components of a random vector $\textbf{X} = (X_1,...,X_d)$ that are consistent against any alternative to independence and require no distributional knowledge or assumptions. While such "blanket" non-parametric tests have been available since the seminar work of \cite{Blum1961a}, the tests considered here draw upon new procedures that recently emerged in the literature on empirical copula processes and stem from the margin-free characterization of independence attainable through copula. 

To begin, let $F$ represent the joint d.f. of $\textbf{X}$ and $F_1,...,F_d$ be the marginals, which we assume to be unknown. Following the result of \cite{sklar}, the joint d.f. $F$ can be written as
\begin{equation}
F(\textbf{X}) = C[F_1(x_1),...,F_d(x_d)],\hspace{4mm} (x_1,...,x_d) \in \mathbb{R}^d,
\end{equation} where the function $C:[0,1]^d \rightarrow [0,1]$ is the so-called \textit{copula} of $\textbf{X}$. The copula has become central to the analysis of dependence as it provides a complete, and in the case of continuous marginals, a unique description of the relationship between $X_1,...,X_d$. Many, if not all non-parametric measures of dependence can be viewed as functions of $C$. For example, when $d=2$ rank correlation measures such as Kendall's $\tau$ and Spearman's $\rho$ can be written in terms of $C$ as 
\begin{equation}
\tau = 4 \int_{[0,1]^2} C(u_1,u_2)dC(u_1,u_2) -1,
\end{equation} and
\begin{equation}
\rho = 12 \int_{[0,1]^2} u_1u_2 dC(u_1,u_2) -3.
\end{equation} For details and a comprehensive introduction to copulas see \cite{Nelsen2006}. 

While non-parametric tests for independence based on linear or rank-correlation measures such as $\tau$ and $\rho$ are easy to implement, they are not consistent against any general alternative to independence and may lack power particularly when dependence between the components of $\textbf{X}$ is non-monotone. In a seminar paper \cite{Blum1961a} propose a portmanteau test for independence based on empirical distribution functions $\hat{F}_1,..,\hat{F}_d$. They key limitation of the test of \cite{Blum1961a}, however, is that the limiting distribution of the statistic depends on the specification of the marginals $F_1,..,F_d$. 

It is easy to see that the hypothesis of independence can be characterized by  \textit{independence copula} $C^\perp(\textbf{u}) = \prod_{j=1}^d u_j$, $\textbf{u} \in [0,1]^d$, which makes it natural to develop tests of independence based on empirical process
\begin{equation}
\label{empirical_copula_process}
\sqrt{n} \left [ C_n(\textbf{u}) - \prod_{j=1}^d u_j \right ], \hspace{4mm} \textbf{u} \in [0,1]^d,
\end{equation} where $C_n$ is an estimate of $C$ obtained from $n$ independent copies $(X_{1,1},...,X_{1,d})$, ...,$(X_{n,1},...,X_{n,d}))$. For a fixed $n$, the random function in (\ref{empirical_copula_process}) shows the distance between the dependence structure of $\textbf{X}$ encoded in $C$ and independence characterized by $C^\perp$. Note that since for any $F_1,..,F_d$ random variables $u_j = F_j(x_j)$ are uniform on $[0,1]$, the behavior of empirical copula process and  hence of functionals of (\ref{empirical_copula_process}) is completely independent of the marginals.

 Inspired by \cite{hoeffding} and \cite{dugue1975}, a portmanteau test for independence based on M\"{o}bius decomposition of the empirical copula process into independent sub-processes was proposed in \cite{Deheuvels1980}. Finite-sample behaviour of Cram\'{e}r-von Mises functionals of such sub-processes was studied in \cite{Genest2004} both in serial and non-serial settings, and \cite{Genest2006} obtain the power curve for the test of \cite{Genest2004} in bivariate case and compare it against the power of alternative procedures based on linear rank statistics. The asymptotic efficiency of a Cram\'{e}r-von Mises test based on (\ref{empirical_copula_process}) is investigated in \cite{Genest2007a}. More recently, \cite{Kojadinovic2009b} generalize the results of \cite{Deheuvels1980} to the case of random vectors and derive asymptotic behaviour of corresponding Cram\'{e}r-von Mises statistics for vectorial independence, with an extension of the test of \cite{Kojadinovic2009b} to vector time series provided in \cite{Kojadinovic2009d}. An important application where a test for independence is used to probe for goodness of fit of Archimedian copulas can be found in \cite{Quessy2010}.

The objective of this paper is to present a highly-flexible class of non-parametric tests for independence based on \textit{weighted} Cram\'{e}r-von Mises functionals of  (\ref{empirical_copula_process}). The weights act as a tuning parameter which, as we show, enables adjustment of power properties of the test and makes it more sensitive towards certain types of dependence. The first application of weights in the context of copula tests for independence appeared recently in \cite{Deheuvels2007}, where the use of an exponential weighting function is suggested, but all practical issues such as computation of the test statistic, its critical points, and most importantly, the effect of weights on test power are left unexplored. 

The aim of this paper is to fill this gap. First, we state the asymptotic properties of weighted statistics in the general case, allowing for weights to be defined by any arbitrary bounded and integrable function. This leads to a wide class of consistent weighted tests for independence and enables easy switching of the weights. Second, we conduct a simulation study to assess, for the first time, the impact of weights on the power of copula test for independence under a variety of alternatives and find that weights that are proportional to the deviations of $C$ from independence copula appear to lead to greatest power gains. We provide a discussion of the weight choice problem and give additional results that simplify the computation weighted statistics from sample ranks.

The rest of this paper is organized as follows. The second section introduces the generalized weighted Cram\'{e}r-von Mises independence statistic, states its asymptotic properties and discusses computation of the statistic from sample ranks. The third section focuses on issues associated with the choice of the weights, and provides closed-form expressions for several weighted statistics likely to be interesting from practical standpoint. Finally, a simulation study is conducted in the fourth section to assess the impact of weights on test power under variety of dependence alternatives. A brief discussion of the results is provided in Section \ref{section::discussion}.

\section{Weighted Cram\'{e}r-von Mises tests for independence}
\label{section::weighted_tests}

We begin by reviewing asymptotic properties of the empirical copula process, form which the limiting behavior of the statistics will later follow.

\subsection{The empirical copula process}
 The empirical copula $C_n$ is usually defined as 
\begin{equation}
C_n(\textbf{u}) = \frac{1}{n} \sum_{i=1}^n \prod_{j=1}^d \mathbb{I} [\hat{F}_{j}(X_{i,j}) \leq u_j],\hspace{4mm} \textbf{u} \in [0,1]^d,
\end{equation} where $\mathbb{I}[\cdot]$ is an indicator function and $\hat{F}_{j}(x)$ is the empirical c.d.f. of $X_j$, for $j=1..d$, given by
\begin{equation}
\hat{F}_{j}(x) = \frac{1}{n}\sum_{i=1}^n \mathbb{I}(X_{i,j} \leq x), \hspace{4mm} x \in \mathbb{R}.
\end{equation}  Letting $\hat{U}_{i,j} = \hat{F}_j(X_{i,j})$, the empirical copula can be viewed simply as the empirical c.d.f. of percentile ranks $\hat{U}_{i,j}$:
\begin{equation}
C_n(\textbf{u}) = \frac{1}{n}\sum_{i=1}^{n} \prod_{j=1}^d \mathbb{I}(\hat{U}_{i,j} \leq u_j),\hspace{4mm} \textbf{u}\in[0,1]^d.
\end{equation} The estimator $C_n(\textbf{u})$ appears to first have been studied in \cite{deheuvals1979}, and its asymptotic properties are usually established through the empirical copula process.

Let $\mathbb{D}$ represent the space of all bounded functions from $[0,1]^d \rightarrow \mathbb{R}$ equipped with uniform metric. The following theorem establishes asymptotic properties of the empirical copula process and of $C_n(\textbf{u})$. These results appear in \cite{Fermanian2004} and \cite{tsukahara} and are summarized in Theorem 1 of \cite{Kojadinovic2009b}, which we restate without proof, with the requirements on $C$ refined as in \cite{Segers2012}:
\begin{theorem}
\label{theorem::empirical_copula_process}
Suppose that for each $j \in \{1,..,d\}$, the partial derivatives $\partial C/\partial u_j$ exist and are continuous on the set $V_{j}:=\{\textbf{u} \in [0,1]^d:0 < u_j < 1\}$. Then, the empirical process
\begin{equation}
\sqrt{n} \left [ C_n(\textbf{u}) - C(\textbf{u}) \right ], \hspace{4mm}\textbf{u} \in [0,1]^d
\end{equation} converges weakly in $\mathbb{D}$ to tight centred Gaussian process 
\begin{equation}
\mathcal{G}(\textbf{u}) = \mathcal{B}(\textbf{u}) - \sum_{i=1}^d \partial_i C(\textbf{u})\mathcal{B}(1,...,1,u_i,1,...,1),\hspace{4mm} \textbf{u} \in [0,1]^d,
\end{equation} where $\partial_i C(\textbf{u})$ is the partial derivative of copula $C$ with respect to its i'th component and $\mathcal{B}$ is a multivariate tied-down Brownian bridge on $[0,1]^d$ with covariance function 
\begin{equation}
E[\mathcal{B}(\textbf{u})\mathcal{B}(\textbf{u}')] = C(\textbf{u} \wedge \textbf{u}') - C(\textbf{u})C(\textbf{u}').
\end{equation}
\end{theorem}

The limiting behavior of (\ref{empirical_copula_process}) under the hypothesis of independence follows directly from Theorem \ref{theorem::empirical_copula_process}:

\begin{corollary}
\label{corollary::independence_empirical_copula_process}
Under the null of independence, the empirical copula process converges weakly in $\mathbb{D}$ to tight centred Gaussian process 

\begin{equation}
\mathcal{M}(\textbf{u}) = \mathcal{B}(\textbf{u}) - \sum_{i=1}^d \left (\prod_{j=1,j\neq i}^d u_j \right ) \mathcal{B}(1,..,1,u_i,1,..,1),
\end{equation} for $\textbf{u} \in [0,1]^d$, where $\mathcal{B}$ is multivariate tied-down Brownian bridge on $[0,1]^d$ with covariance function 
\begin{equation}
\label{independence_process_B_covariance}
E[\mathcal{B}(\textbf{u})\mathcal{B}(\textbf{u}')] = \prod_{j=1}^d \min\{u_j,u'_j\} - \prod_{j=1}^d u_ju'_j.
\end{equation}
\end{corollary} The process $\mathcal{M}(\textbf{u})$ is a multidimensional completely tucked Brownian sheet, with the value of $\mathcal{M}(\textbf{u})$ along the boundaries of the $[0,1]^d$ almost surely zero. For additional details see Example 3.8.2 in \cite{wellner}.

\subsection{Weighted Cram\'{e}r-von Mises statistic}

The focus in this paper is on tests based on weighted Cram\'{e}r-von Mises functionals of (\ref{empirical_copula_process}) defined as
\begin{equation}
\label{weighted_statistic}
W_n = n \int_{[0,1]^d} \left ( C_n(\textbf{u}) - \prod_{j=1}^d u_j \right ) ^2 w(\textbf{u}) d\textbf{u},\hspace{4mm}\textbf{u} \in [0,1]^d,
\end{equation} where $w(\textbf{u})$ is an arbitrary function that weighs the deviations of $C_n$ from $C^\perp$ throughout the support $[0,1]^d$. The objective of $w(\textbf{u})$ is to shift emphasis between the different parts of the support of the joint distribution of $X_1,..,X_d$ giving the test added flexibility, and as we show in Section \ref{section::test_power_simulation}, making the statistic more sensitive to certain types of dependence. The ability to manipulate test power through the choice of the weights has added significance in many settings: for example, when independence tests are used to test for goodness of fit as in \cite{Quessy2010}.

\subsection{Asymptotic properties of weighted statistics}

Several regularity requirements need to be placed on the weighting function $w(\textbf{u})$ to ensure that the integral in (\ref{weighted_statistic}) exist. Specifically, we assume that $w(\textbf{u})$ lies in $\mathbb{D}$, is non-negative and integrable on $[0,1]^d$. Theorem \ref{theorem::convergence_in_distribution} characterizes the limiting distribution of $W_n$ under independence in the general case.

\begin{theorem}
\label{theorem::convergence_in_distribution}
Suppose that for each $j \in \{1,..,d\}$, the partial derivatives $\partial C/\partial u_j$ exist and are continuous on the set $V_{j}:=\{\textbf{u} \in [0,1]^d:0 < u_j < 1\}$. Then, under mutual independence of $X_1,..,X_d$, for any integrable $w(\textbf{u}) \in \mathbb{D}$, the statistic $W_n$ converges in distribution to
\begin{equation}
W = \int_{[0,1]^d} \mathcal{M}(\textbf{u})^2 w(\textbf{u}) d\textbf{u},\hspace{4mm}\textbf{u} \in [0,1]^d,
\end{equation} where $\mathcal{M}(\textbf{u})$ is Brownian sheet as defined in Corollary \ref{corollary::independence_empirical_copula_process}.
\end{theorem}
\begin{proof}
For any $w\in \mathbb{D}$ and $\textbf{u}\in[0,1]^d$, let $\psi(f)(\textbf{u}):\mathbb{D} \rightarrow \mathbb{D}$ be a map defined as $\psi(f)(\textbf{u})=f(\textbf{u})\sqrt{w(\textbf{u})}$. Let $g_n(\textbf{u}) = \sqrt{n}[C_n(\textbf{u}) - C^\perp(\textbf{u})]$. For any sequence $f_n$, $n=1,2,..$ in $\mathbb{D}$ s.t. $f_n \rightarrow f$, we have that 
\begin{equation}
\sup_{\textbf{u}} | \psi(f_n)(\textbf{u}) - \psi(f)(\textbf{u}) | \leq \sup_{\textbf{u}} |f_n(\textbf{u}) - f(\textbf{u})| \sup_{\textbf{u}} |w(\textbf{u})| \rightarrow 0,
\end{equation}
 since $w(\textbf{u})$ is bounded in absolute value. Then, $\psi(f_n)(\textbf{u}) \rightarrow \psi(f)(\textbf{u})$, and continuous mapping theorem implies that $\psi(g_n)(\textbf{u}) = \sqrt{n}[C_n(\textbf{u}) - C^\perp(\textbf{u})]\sqrt{w(\textbf{u})}$ converges weakly in $\mathbb{D}$ to $\psi(\mathcal{M}(\textbf{u})) = \mathcal{M}(\textbf{u}) \sqrt{w}(\textbf{u})$, and convergence in distribution to $W$ follows.
\end{proof} Boundedness of $w(\textbf{u})$ and of the covariance function of $\mathcal{M}(\textbf{u})$ ensure that the limiting distribution of $W_n$ is non-degenerate. It is also clear that the distribution of $W$ does not depend on neither the marginal not the joint distributions of $X_j$'s, meaning that the asymptotic critical values for (\ref{weighted_statistic}) can be easily tabulated, given $w(\textbf{u})$ and $d$.

\subsection{Computation from sample ranks}
\label{section::integrability_and_computation}

For a set $w(\textbf{u})$, computation of the test statistic in (\ref{weighted_statistic}) is only as difficult as the integration of the weighting function. To see this, consider maps $\mu_1(w)(\textbf{a})$, $\mu_2(w)(\textbf{a})$ and $\mu_3(w)$ from $\mathbb{D}$ onto $\mathbb{D}$ given by
\begin{align}
\mu_1(w)(\textbf{a}) &= \int_{a_1}^1..\int_{a_d}^1 w(u_1,..,u_d) du_1..du_d,\hspace{4mm}\textbf{a} \in [0,1]^d, \\
\mu_2(w)(\textbf{a}) &= \int_{a_1}^1..\int_{a_d}^1 u_1u_2..u_d w(u_1,..,u_d)du_1..du_d,\hspace{4mm}\textbf{a} \in [0,1]^d, 
\end{align} and lastly
\begin{equation}
\mu_3(w) = \int_0^1 .. \int_0^1 u_1^2u_2^2..u_d^2 w(u_1,..,u_d)du_1..du_d.
\end{equation} Assumptions on $w(\textbf{u})$ imply that $\mu_1(w)(\textbf{a})$, $\mu_2(w)(\textbf{a})$ and $\mu_3(w)$ exist for any $\textbf{a} \in [0,1]^d$. Letting $\hat{U}_{i,j} = \hat{F}_j(X_{i,j})$ as before, straightforward calculation immediately yields closed-form expression of $W_n$ in terms of $\hat{U}_{i,j}$ and $\mu_1(w)(\textbf{a})$, $\mu_2(w)(\textbf{a})$ and $\mu_3(w)$:

\begin{proposition}
\label{general_closed_form}
We have
\begin{align}
W_n &= \sum_{i=1}^n \left [ \frac{1}{n} \sum_{l=1}^n \mu_1(w)(\hat{U}_{i,1}\vee\hat{U}_{l,1},..,\hat{U}_{i,d}\vee\hat{U}_{l,d}) - 2\mu_2(w)(\hat{U}_{i,1},..,\hat{U}_{i,d}) \right ] \\
&+ n \mu_3(w).
\end{align} 
\end{proposition}
\begin{proof}
\begin{align}
W_n &= n \int_{[0,1]^d} \left ( C_n(\textbf{u}) - C^\perp(\textbf{u}) \right ) ^2 w(\textbf{u}) du \\
&= n\int_{[0,1]^d} C_n(\textbf{u})^2 w(\textbf{u}) du - 2n\int_{[0,1]^d} C_n(\textbf{u}) C^\perp(\textbf{u}) w(\textbf{u}) du \\
&+ n\int_{[0,1]^d} C^\perp(\textbf{u})^2 w(\textbf{u}) du.
\end{align}
For the first term, we have
\begin{align}
n\int_{[0,1]^d}&C_n(\textbf{u})^2 w(\textbf{u}) du \\
&= \frac{1}{n}\int_{[0,1]^d} \left (\sum_{i=1}^n \prod_{j=1}^d \mathbb{I}(\hat{U}_{i,j} \leq u_j) \right )^2 w(\textbf{u}) du \\
&= \frac{1}{n}\sum_{i=1}^n \sum_{l=1}^n \int_{[0,1]^d} \prod_{j=1}^d \mathbb{I}(\hat{U}_{i,j} \leq u_j)\mathbb{I}(\hat{U}_{l,j} \leq u_j) w(\textbf{u}) du \\
&= \frac{1}{n}\sum_{i=1}^n \sum_{l=1}^n \int_{\hat{U}_{i,1}\vee\hat{U}_{l,1}}^1 ... \int_{\hat{U}_{i,d}\vee\hat{U}_{l,d}}^1 w(\textbf{u}) du.
\end{align} Similarly, for the second term,
\begin{align}
2n\int_{[0,1]^d}&C_n(\textbf{u}) C^\perp(\textbf{u}) w(\textbf{u}) du\\
&=2 \int_{[0,1]^d} \left (\sum_{i=1}^n \prod_{j=1}^d \mathbb{I}(\hat{U}_{i,j} \leq u_j) u_j \right ) w(\textbf{u})   du \\
&= 2 \sum_{i=1}^n \int_{[0,1]^d} \prod_{j=1}^d \mathbb{I}(\hat{U}_{i,j} \leq u_j)  u_j w(\textbf{u}) du \\
&= 2\sum_{i=1}^n \int_{\hat{U}_{i,1}}^1 ..\int_{\hat{U}_{i,d}}^1 \prod_{j=1}^d u_j w(\textbf{u}) du.
\end{align} Lastly, substituting $C^\perp(\textbf{u}) = \prod_{j=1}^d u_j$, we get the third term. Grouping the terms and substituting the expressions for $\mu_1$, $\mu_2$ and $\mu_3$ yields the desired result.
\end{proof} For any integrable $w(\textbf{u})\in \mathbb{D}$ , computational formula for $W_n$ can be obtained by deriving expressions for $\mu_1(\textbf{a})$, $\mu_2(\textbf{a})$ and $\mu_3$. For example, consider setting $w(\textbf{u}) = \prod_{j=1}^d u_j^{2\beta_j}$, for $\beta = (\beta_1,...,\beta_d) \in \mathbb{R}^d$, which yields the weighted statistic of \cite{Deheuvels2007}, which we denote by $D_n$:
\begin{equation}
\label{equation::deheuvels_statistic}
D_n = n \int_{[0,1]^d} \left (C_n(\textbf{u}) - \prod_{j=1}^d u_j \right ) ^2 \prod_{j=1}^d u_j^{2\beta_j} d\textbf{u}, \hspace{4mm} \textbf{u} \in [0,1]^d,
\end{equation} The computational formula for $D_n$ in terms of $\hat{U}_{i,j}$, $i=1,..,n$, $j=1,..,d$ follows directly from Proposition \ref{general_closed_form}:
\begin{proposition}
\label{deheuvels_closed_form}
When $w(\textbf{u}) = \prod_{j=1}^d u_j^{2\beta_j}$, for $\beta = (\beta_1,...,\beta_d) \in \mathbb{R}^d$, we have that
\begin{align}
D_n &= \sum_{i=1}^n \left [ \frac{1}{n} \sum_{l=1}^n \prod_{j=1}^d \frac{(1-\hat{U}_{i,j}\vee \hat{U}_{l,j})^{2\beta_j +1}}{2\beta_j +1} - 2\prod_{j=1}^d \frac{(1-\hat{U}_{i,j})^{2\beta_j +2}}{2\beta_j +2} \right ] \\
&+ n \prod_{j=1}^d (2\beta_j +3)^{-1}.
\end{align}
\end{proposition}
\begin{proof}
When $w(\textbf{u}) = \prod_{j=1}^d u_j^{2\beta_j}$, for $\beta = (\beta_1,...,\beta_d) \in \mathbb{R}^d$, using straightforward integration we have that, for $\textbf{a} \in [0,1]^d$,
\begin{align}
\mu_1(\textbf{a}) &= \int_{a_1}^1..\int_{a_d}^1 \prod_{j=1}^d u_j^{2\beta_j} du = \prod_{j=1}^d \frac{(1-a_j)^{2\beta_j + 1}}{2\beta_j + 1}, \\
\mu_2(\textbf{a}) &= \int_{a_1}^1..\int_{a_d}^1 \prod_{j=1}^d u_j^{2\beta_j + 1} du = \prod_{j=1}^d \frac{(1-a_j)^{2\beta_j + 2}}{2\beta_j + 2}, \\
\mu_3 &= \int_{[0,1]^d} \prod_{j=1}^d u_j^{2\beta_j + 2} du = \prod_{j=1}^d \frac{1}{2\beta_j + 3}.
\end{align} Substituting for $\mu_1$, $\mu_2$ and $\mu_3$ into the expression from Proposition \ref{general_closed_form} we get the desired result.
\end{proof} 
Similarly, repeating the derivation with $w(\textbf{u}) = 1$, $\textbf{u} \in [0,1]^d$, easily yields the expression of uniformly-weighted statistic, which we denote by $U_n$, given by:
\begin{equation}
U_n = n \int_{[0,1]^d} \left ( C_n(\textbf{u}) - \prod_{j=1}^d u_j \right) ^2 d\textbf{u},\hspace{4mm}\textbf{u} \in [0,1]^d,
\end{equation} in terms of percentile ranks:
\begin{proposition}
When w(\textbf{u}) = 1, for $\textbf{u} \in [0,1]^d$, we have that
\begin{align}
U_n&= \sum_{i=1}^n \left [ \frac{1}{n} \sum_{l=1}^n \prod_{j=1}^d (1-\hat{U}_{i,j}\vee\hat{U}_{l,j}) - 2\prod_{j=1}^d \frac{1}{2}(1-\hat{U}_{i,j}^2)  \right ] \\
&+ n 3^{-d}.
\end{align}
\end{proposition}
\begin{proof}
For $w(\textbf{u})=1$, $\textbf{u}\in[0,1]^d$, we have that, for $\textbf{a} \in [0,1]^d$,
\begin{align}
\mu_1(\textbf{a})& = \prod_{j=1}^d (1-a_j),\\
\mu_2(\textbf{a})& = \prod_{j=1}^d \frac{1}{2} (1-a_j^2),
\end{align} and $\mu_3 = 3^{-d}$. The rest follows from Proposition \ref{general_closed_form}.
\end{proof}

\section{The choice of the weights}
\label{section::weights_choice_problem}

The general nature of the requirements imposed on the weighting function by Theorem \ref{theorem::convergence_in_distribution} implies that a wide range of statistics given by (\ref{weighted_statistic}) and defined by the choice of $w(\textbf{u})$ is possible. This raises a natural question about the existence, or non-existence of $w(\textbf{u})$ that may be optimal under some sequence of alternatives. The issue of optimality is left for future work; the aim of this section is to explore the effect of weights on test power under a wide variety of alternatives and to provide some guidance with the selection of $w(\textbf{u})$ and future search for optimal weights. We review some issues surrounding the choice of the weights and propose five weighted statistics that correspond to copula models most-commonly encountered in practice. Computational formulas and asymptotic critical values for the five statistics are also provided.

\subsection{Anderson-Darling weights}

The first addition of weights into a Cram\'{e}r-von Mises-type statistic is perhaps due to \cite{Anderson1952a} who use weights to add flexibility to a goodness of fit test for empirical distribution functions. \cite{Anderson1952a} propose setting the weights so that to give the integrated stochastic process unit variance throughout its domain. Given the prominence of the statistic of \cite{Anderson1952a}, it is natural to begin with equivalent variance-driven weights based on $\mathcal{M}(\textbf{u})$, which we consider next. For notational convenience, for any $\textbf{u} \in [0,1]^d$, let $\textbf{u}_{\{i\}} = (1,..,u_i,..,1)\in[0,1]^d$. From Corollary \ref{corollary::independence_empirical_copula_process} it is easy to verify that under independence, the variance function of the process $\mathcal{M}(\textbf{u})$ can be written as
\begin{align}
\label{unweighted_process_variance}
&E[\mathcal{M}(\textbf{u})^2]=E\left[\left(\mathcal{B}(\textbf{u}) - \sum_{i=1}^d \left (\prod_{j=1,j\neq i}^d u_j \right ) \mathcal{B}(u_{\{i\}})\right)^2\right] \\
&=\prod_{j=1}^d u_j(1-u_j) - 2\sum_{i=1}^d \left (\prod_{j=1,j\neq i}^d u_j \right )\left (\prod_{j=1}^d u_j \right )(1-u_i) \\
&+\sum_{i=1}^d \left ( \prod_{j=1,j\neq i}^d u_j^2 \right ) (u_i - u_i^2),
\end{align} for $\textbf{u} \in [0,1]^d$. Letting the weights $w_a(\textbf{u})$ be equal to the reciprocal of the variance, we have that $E[\mathcal{M}(\textbf{u})^2 w_a(\textbf{u})] = 1$, $\forall \textbf{u} \in [0,1]^d$. We refer to such weights as Anderson-Darling weights, and the corresponding weighted test statistic as Anderson-Darling independence statistic. Interestingly, here such Anderson-Darling weights do not satisfy integrability requirements set out in the previous section, meaning that the corresponding Anderson-Darling independence statistic based on the empirical copula process does not exist. To see this, consider a case when $d=2$, in which Anderson-Darling weights are given by
\begin{equation}
w_a(\textbf{u}) = \left [u_1u_2(u_1-1)(u_2-1) \right]^{-1},\hspace{4mm}\textbf{u} \in [0,1]^2.
\end{equation} It is easy to verify that for any $\textbf{a} \in [0,1]^d$, the corresponding $\mu_1(\textbf{a})$ is infinite, meaning that the integral in (\ref{weighted_statistic}) does not exist. 

The use of Anderson-Darling weights with an offset to ensure integrability as in \cite{Genest2012} can provide a solution to this problem, and represents scope for future work.

\subsection{Median weights}
In practice, weight selection may be motivated by specific interest in certain types of dependence. For example, when the application is such that dependence among observations closer to the median of the distribution is of interest and outliers in the data are of lesser importance, a weighting function which assigns a lower weight to the tails may be used. 

Among the better-known copulas which posses this property are Gaussian, Ali-Mikhail-Haq and Frank families which deviate from independence copula the most around the median and least in the tails. To this end, setting weights to $w_m(\textbf{u}) = \prod_{j=1}^d u_j(1-u_j)$, for $\textbf{u}\in[0,1]^d$ will place greater emphasis on observations around the median of the distribution where largest deviations occur, which leads to a weighted statistic given by
\begin{equation}
M_n = n \int_{[0,1]^d} \left (C_n(\textbf{u}) - \prod_{j=1}^d u_j \right )^2 \prod_{j=1}^d u_j(1-u_j)du,\hspace{4mm}\textbf{u} \in [0,1]^d.
\end{equation} The function $w_m(\textbf{u})$ is bounded and integrable, meaning that asymptotic properties of $M_n$ follow directly from Theorem \ref{theorem::convergence_in_distribution}. As before, computational formula for $M_n$ can be obtained using Proposition \ref{general_closed_form}.
\begin{proposition}
We have that
\begin{align}
M_n &=\sum_{i=1}^n \left [ \frac{1}{n} \sum_{l=1}^n \prod_{j=1}^d \frac{1}{6}(2(\hat{U}_{i,j}\vee\hat{U}_{l,j})+1)(\hat{U}_{i,j}\vee\hat{U}_{l,j}-1)^2 \right.\\
&\left. - 2\prod_{j=1}^d\left (\frac{1}{4}(\hat{U}_{i,j}^4 -1)+\frac{1}{3}(1-\hat{U}_{i,j}^3) \right ) \right ] + \frac{n}{20^d}.
\end{align}
\end{proposition}
\begin{proof}
When $w(\textbf{u}) = \prod_{j=1}^d u_j(1-u_j)$, for $\textbf{u}\in[0,1]^d$, we have that, for $a \in [0,1]^d$,
\begin{align}
\mu_1(\textbf{a})&=\prod_{j=1}^d \frac{1}{6}(2a_j+1)(1-a_j)^2,\\
\mu_2(\textbf{a})&=\prod_{j=1}^d\left (\frac{1}{4}(a_j^4 -1)+\frac{1}{3}(1-a_j^3) \right ),
\end{align} and $\mu_3 = 20^{-d}$. Applying Proposition \ref{general_closed_form} yields the desired result.
\end{proof} Another argument for the use of $w_m(\textbf{u})$ or similar median weights may stem from the constraints imposed on $C$ by the theory of distribution functions. Following the results of \cite{Frechet1952} and \cite{Hoeffding1940}, for any copula $C$ and for any $\textbf{u} \in [0,1]^d$, we have that $\max(u_1+u_2+..+u_d-d+1,0) \leq C(\textbf{u}) \leq \min(u_1,u_2,..,u_d)$. The functions $M(\textbf{u}) = \min(u_1,u_2,..,u_d)$, $\textbf{u} \in [0,1]^d$, and $W(\textbf{u}) = \max(u_1+u_2+..+u_d-d+1,0)$, $\textbf{u} \in [0,1]^d$, are the copula Frech\'{e}t-Hoeffding bounds, and are also copulas when $d=2$. Rearranging the inequality we have that $(C(\textbf{u}) - C^\perp(\textbf{u}))^2 \leq \max((W(\textbf{u})-C^\perp(\textbf{u}))^2,(M(\textbf{u})-C^\perp(\textbf{u}))^2)$, for any $\textbf{u}\in[0,1]^d$, meaning that the maximum amount by which the copula $C$ can deviate form independence copula $C^\perp$ varies greatly across $[0,1]^d$. Examining the right-hand side of the inequality will show that this distance is maximal precisely at the median of the distribution and is decreasing towards the tails. In that sense, weighting function $w_m(\textbf{u})$ places greater emphasis on the region where scope for deviations from independence is greatest. 
\subsection{Symmetric tail weights}
Another type of dependence which is important in applications such as risk management is tail dependence, which refers to the tendency of extreme values to be associated. Symmetric tail dependence occurs when the extremes are related regardless of their sign; for example, an extreme realization of one variable may indicate, with equal likelihood, a higher chance of observing either an extremely-large or an extremely-small value of another. Such dependence may exist in the complete absence of linear or rank correlation, making it difficult to detect. Distributions constructed using Student's t copula feature symmetric tail dependence, with the copula deviating from independence equally in all tails. When such dependence among all outliers is of interest, setting the weights to $w_t(\textbf{u}) = \prod_{j=1}^d (u_j - 0.5)^2$, $\textbf{u}\in[0,1]^d$, yields tail-weighted statistic $T_n$ given by
\begin{equation}
T_n = n \int_{[0,1]^d} \left (C_n(\textbf{u}) - \prod_{j=1}^d u_j\right)^2 \prod_{j=1}^d (u_j - 0.5)^2 du,\hspace{4mm}\textbf{u} \in [0,1]^d.
\end{equation} As before, asymptotic properties follow directly from Theorem \ref{theorem::convergence_in_distribution}, and computation formula from Proposition \ref{general_closed_form}.
\begin{proposition}
We have that
\begin{align}
T_n &= \sum_{i=1}^n \left [ \frac{1}{n} \sum_{l=1}^n \prod_{j=1}^d (24^{-1} -3^{-1}(\hat{U}_{i,j} \vee \hat{U}_{l,j} - 2^{-1})^3 ) \right. \\
&\left.- 2\prod_{j=1}^d (24^{-1} - 4^{-1}\hat{U}_{i,j}^4 + 3^{-1}\hat{U}_{i,j}^3 - 8^{-1}\hat{U}_{i,j}^2) \right ] + n 30^{-d}.
\end{align}
\end{proposition}
\begin{proof}
When $w_t(\textbf{u}) = \prod_{j=1}^d (u_j - 0.5)^2$, $\textbf{u}\in[0,1]^d$, by straightforward integration, for any $a\in[0,1]^d$, we get
\begin{align}
\mu_1(\textbf{a}) &= \prod_{j=1}^d (24^{-1} -3^{-1}(a_j - 2^{-1})^3 ),\\
\mu_2(\textbf{a}) &=\prod_{j=1}^d (24^{-1} - 4^{-1}a_j^4 + 3^{-1}a_j^3 - 8^{-1}a_j^2),
\end{align} and $\mu_3 = 30^{-d}$. The rest follows from Proposition \ref{general_closed_form}. 
\end{proof} A more-detailed examination will reveal a degree of similarity between $w_t(\textbf{u})$ and the Anderson-Darling weights, meaning that $T_n$ represents a step towards the unattainable Anderson-Darling independence statistic.
\subsection{Upper-tail weights}
The concept of tail dependence can be narrowed to describe dependence at a particular distribution quadrant. Upper-tail dependence occurs when data cluster in the upper-right corner of the joint distribution, and is measured by upper-tail dependence coefficient $\lambda_U$. Formally, for $d=2$, the upper-tail dependence coefficient is defined as
\begin{eqnarray}
\lambda_U = \lim_{t\rightarrow 1^{-1}}P[X_1 > F_1^{-1}(t) | X_2 > F_2^{-1}(t)].
\end{eqnarray} Lower-tail dependence is defined similarly, with the corresponding  coefficient $\lambda_L$ given by
\begin{equation}
\lambda_L = \lim_{t\rightarrow0^{+}} P[X_1 \leq F_1^{-1}(t) | X_2 \leq F_2^{-1}(t)].
\end{equation} These parameters depend only on the copula of $X_1,..,X_d$ and can be expressed in terms of $C$; for additional details see Theorem 5.4.2 in \cite{Nelsen2006}. The concept of upper- and lower-tail dependence is relevant to many applications, and the presence of asymmetric tail-dependence in various data is well-documented; for financial examples see \cite{Michelis2010} and \cite{Patton2006}. Among the better-known copulas which feature upper, but not lower-tail dependence are Gumbel, Joe, Caudras-Aug\'{e} and Marshall-Olkin families. When upper-tail dependence is of interest, setting $w_p(\textbf{u}) = \prod_{j=1}^d u_j^2$, $\textbf{u}\in [0,1]^d$ will place greater emphasis on observations in the first quadrant of the distribution and lead to an upper tail-weighted statistic $P_n$ given by
\begin{equation}
P_n = n \int_{[0,1]^d} \left ( C_n(\textbf{u}) - \prod_{j=1}^d u_j \right )^2 \prod_{j=1}^d u_j^2 du,\hspace{4mm}\textbf{u}\in[0,1]^d.
\end{equation} Note that the statistic $P_n$ is a special case of (\ref{equation::deheuvels_statistic}) which arises when the exponent coefficients are set to $\beta_j = 1$, for $j=1..d$. Computational formula for $P_n$ follows directly from Proposition \ref{deheuvels_closed_form}, and asymptotic properties from Theorem \ref{theorem::convergence_in_distribution}.
\begin{proposition}
We have that
\begin{align}
P_n &=  \sum_{i=1}^n \left [ \frac{1}{n} \sum_{l=1}^n \prod_{j=1}^d \frac{(1-\hat{U}_{i,j}\vee \hat{U}_{l,j})^{3}}{3} - 2\prod_{j=1}^d \frac{(1-\hat{U}_{i,j})^{4}}{4} \right ] + n 5^{-d}.
\end{align}
\end{proposition}
\subsection{Lower-tail weights}
Copulas belonging to Clayton and Raftery families are among those which have lower, but not upper-tail dependence. When the focus is on dependence among small outliers, setting the weights to $w_l = \prod_{j=1}^d (1-u_j)^2$, $\textbf{u}\in[0,1]^d$, will place greater emphasis on the lower tail of the distribution, leading to a lower tail-weighted statistic $L_n$ given by
\begin{equation}
L_n = n \int_{[0,1]^d} \left ( C_n(\textbf{u}) - \prod_{j=1}^d u_j \right )^2 \prod_{j=1}^d (1-u_j)^2 du,\hspace{4mm}\textbf{u}\in[0,1]^d.
\end{equation} As before, the expression of $L_n$ in terms of percentile ranks can be obtained using Proposition \ref{general_closed_form}, and asymptotic properties follow from Theorem \ref{theorem::convergence_in_distribution}.
\begin{proposition}
We have that
\begin{align}
L_n &=\sum_{i=1}^n \left [ \frac{1}{n} \sum_{l=1}^n \prod_{j=1}^d 3^{-1}(1-\hat{U}_{i,j}\vee\hat{U}_{l,j})^3\right.\\
&\left. - 2 \prod_{j=1}^d \left (\frac{1}{12} -\frac{1}{4}\hat{U}_{i,j}^3+\frac{2}{3}\hat{U}_{i,j}^4-\frac{1}{2}\hat{U}_{i,j}^2 \right ) \right ] + n 30^{-d}.
\end{align} 
\end{proposition}
\begin{proof}
For any $a\in[0,1]^d$, 
\begin{align}
\mu_1(\textbf{a}) &= \prod_{j=1}^d 3^{-1} (1-a_j)^3,\\
\mu_2(\textbf{a}) &= \prod_{j=1}^d \frac{1}{12}-\frac{1}{4}a_j^4+\frac{2}{3}a_j^3-\frac{1}{2}a_j^2,
\end{align} and $\mu_3 = 30^{-d}$. The rest follows from Proposition \ref{general_closed_form}.
\end{proof} It is worth noting that many copulas can significantly deviate from independence in multiple parts of the distribution at the same time, making the problem of weight selection more complex. The weight choice problem is also unavoidable since even the use of "unweighted" statistics as in \cite{Kojadinovic2009b}, \cite{Kojadinovic2009d} or \cite{Genest2004} represents a choice of $w(\textbf{u})=1$, for any $\textbf{u} \in [0,1]^d$, which is not conceptually different from any other $w(\textbf{u})$. A versatile approach to weight selection could involve a two-step procedure by which the function $w(\textbf{u})$ is estimated from the data. Such data-driven weights represent significant scope for future work.

\section{Simulation study}
\label{section::test_power_simulation}

\subsection{Asymptotic critical values}
To carry out the tests in practice, appropriate critical values need to be obtained. Since the expansion of the distribution of $W$ in terms of known distribution functions appears unavailable (see \cite{Deheuvels2005}), approximate asymptotic critical values for a given $w(\textbf{u})$ and $d$ can be tabulated by simulation. We obtain critical points from the limiting distributions of $U_n$, $M_n$, $T_n$, $P_n$ and $L_n$ based on $100,000$ sample draws for the case when $d=2$, and group them in Table \ref{critical_values_table} below.
\begin{table}[htbp]
\begin{center}
\begin{tabular}{|l|l|l|l|l|l|}
\hline
 &  \multicolumn{ 4}{c|}{\textbf{Prob(statistic$\geq$critical value)}} & \\ \hline
\textbf{Weight} &  \multicolumn{1}{c|}{\textbf{0.15}} & \multicolumn{1}{c|}{\textbf{0.10}} & \multicolumn{1}{c|}{\textbf{0.05}} & \multicolumn{1}{c|}{\textbf{0.01}} & \multicolumn{1}{c|}{\textbf{Scale}} \\ \hline
Uniform ($U_n$)  & $4.4256$  & $5.1237$ & $6.3702$  & $9.3922$ & $\times10^{-2} $  \\ \hline
Median ($M_n$)   & $1.8012$ & $2.1243$ & $2.7126$ & $4.1381$ & $\times10^{-3}$ \\ \hline
Tails ($T_n$)  & $1.2065$ & $1.3402$ & $1.5649$ & $2.0912$ & $\times10^{-4}$ \\ \hline
Upper tail ($P_n$)  & $4.6109$ & $5.3775$ & $6.6872$ & $9.8749$ & $\times10^{-3}$ \\ \hline
Lower tail ($L_n$)  & $3.8026$ & $4.4638$  & $5.6291$  & $8.4140$ & $\times10^{-3}$  \\ \hline
\end{tabular}
\end{center}
\caption{Approximate critical points from asymptotic distributions of $U_n$, $M_n$, $T_n$, $P_n$ and $L_n$.}
\label{critical_values_table}
\end{table}
\subsection{Permutation critical values}
\label{permutation_critical_values}
One disadvantage of asymptotic critical values is the need re-tabulate the points whenever new weights are to be used. Alternatively, critical values from finite-sample distribution can be obtained by following a permutation procedure, conditional on the sample at hand. As before, let $(X_{1,1},...,X_{1,d})$, ...,$(X_{n,1},...,X_{n,d})$ represent $n$ independent copies of $\textbf{X}$. Let $R_j = (R_{1,j},..,R_{n,j})$, $j=1..d$, be uniformly-distributed on the set of permutations $\{1,..,n\}$ such that for any $j \neq k$, $R_j$ is independent from $R_k$ and from $(X_{1,1},...,X_{1,d})$, ...,$(X_{n,1},...,X_{n,d})$. A version of (\ref{empirical_copula_process}) based on $(X_{R_{1,1},1},...,X_{R_{1,d},d})$, ..., $(X_{R_{n,1},1},...,X_{R_{n,d},d})$ is the \textit{permutation independence process}, and we denote weighted test-statistic based on such process by $\tilde{W}_n$. Given level of significance $\alpha$, permutation critical value is defined as $\tilde{c}_n = \inf\{t > 0:P(\tilde{W}_n > t) < \alpha\}$. As noted on p. 371 of \cite{wellner}, the proof of consistency of a test based on $\tilde{c}_n$ appears to be absent from the literature, but the asymptotic behaviour of permutation independence process is likely to be similar to that of the bootstrap independence process; for details see p. 369 of \cite{wellner}. Here, in all simulations we numerically verify consistency and correct nominal size of the tests based on $\tilde{c}_n$. In practice, such permutation-based tests can be carried out using the same procedure as in Section 3.5 of \cite{Kojadinovic2009b}:
\begin{enumerate}
\item Let $W_{n,0}$ be the statistic computed form the original sample.
\item Generate random permutations $R_j$, $j=1..d$, and calculate the value of the statistic $\tilde{W}_{n,1}$ based on the permuted sample.
\item Repeat the previous step $N-1$ times, leading to a collection of permuted statistics $\tilde{W}_{n,2},..,\tilde{W}_{n,N}$.
\item Approximate $p$-value for the test can be obtained as
\begin{equation}
\frac{1}{N+1}\left ( \frac{1}{2} + \sum_{i=1}^N \mathbb{I}(\tilde{W}_{n,i} \geq W_{n,0}) \right )
\end{equation} 
\end{enumerate} The permutation approach described above is also equivalent to the one used in Section 4.4 of \cite{Genest2004}.

\subsection{Simulations}
To assess test performance in finite samples and to document the effect of weights on test power, an extensive simulation study has been  conducted. Due to the availability of computational formulas for all statistics considered here, a single permutation-based test can be carried out in milliseconds using a consumer-grade processor with a high degree of accuracy. The precise estimation of the Type II error rate, on the other hand, requires large repeated sampling. To ensure that a broad range of alternatives can be covered, the complexity in this section is kept to a minimum, and the focus is on the case when $d=2$. In what follows, we use simulations to obtain test power in bivariate samples under a variety of dependence alternatives. In particular, we consider  bivariate distributions with standard normal marginals and Gaussian, $t$, Gumbel, Clayton and Frank copulas: 
\begin{itemize}
\item Bi-variate Gaussian copula is parametrized by the correlation coefficient $\rho \in [0,1]$. Independence occurs when $\rho=0$, and the copula is defined as
\begin{align}
\Phi_{\rho}&(\textbf{u}) =\\
& \int_{-\infty}^{\Phi^{-1}(u_1)}\int_{-\infty}^{\Phi^{-1}(u_2)} \frac{1}{2\pi \sqrt{1-\rho^2}} \exp \left(\frac{2\rho s t - s^2 - t^2}{2(1-\rho^2)} \right) ds dt,
\end{align} for $\textbf{u} \in [0,1]^2$, where $\Phi$ is the univariate standard normal distribution function.
\item Similarly, bi-variate $t$-copula with a linear correlation parameter $\rho \in [0,1]$ and degrees of freedom $k \in (0,\infty)$ is defined as
\begin{align}
T_{\rho,k}&(\textbf{u}) = \\
& \int_{-\infty}^{t_k^{-1}(u_1)}\int_{-\infty}^{t_k^{-1}(u_2)} \frac{1}{2\pi \sqrt{1-\rho^2}} \left ( 1 + \frac{s^2+t^2-2\rho s t }{k(1-\rho^2)} \right )^{\frac{-k+2}{2}} dsdt,
\end{align} for $\textbf{u} \in [0,1]^2$, where $t_k(x) = \int_{-\infty}^x \frac{\Gamma((k+1)/2)}{\sqrt{\pi k}(k/2)}(1+\frac{s^2}{k})^\frac{-k+1}{2}ds$, and $\Gamma$ is the Euler function. For $\rho = 0$, copula $T_{\rho,k}$ approaches independence as $k \rightarrow \infty$, while smaller values of $k$ result in increasing symmetric tail dependence.
\item For $d=2$, Gumbel copula is defined as
\begin{align}
G_{\alpha}(\textbf{u}) = \exp\{-[(-\ln(u_1)^\alpha) + (-\ln(u_2))^\alpha]^{1/\alpha} \},
\end{align} for $\textbf{u} \in [0,1]^2$, where dependence parameter $\alpha \in [1,\infty)$ is such that $\alpha = 1$ implies independence, and increasing values imply greater degree of concordance and upper-tail dependence; in fact, for bi-variate Gumbel copula we have that $\lambda_U = 2-2^{1/\alpha}$.
\item The bi-variate Clayton copula is given by
\begin{align}
C_{\theta}(\textbf{u}) = \max(u_1^{-\theta}+u_2^{-\theta}-1,0)^{-1/\theta}, 
\end{align}$\textbf{u}\in[0,1]^2$, where $\theta \in [-1,\infty) \setminus \{0\}$ is the dependence parameter. Independence occurs when $\theta = 0$, while increasing values of $\theta$ lead to greater concordance and lower-tail dependence; here, we have that for $\theta\geq0$, $\lambda_L = 2^{-1/\theta}$.
\item Lastly, bi-variate Frank copula is defined as
\begin{align}
F_{\gamma}(\textbf{u}) = -\frac{1}{\gamma} \ln \left ( 1 + \frac{(e^{-\gamma u_1}-1)(e^{-\gamma u_2}-1)}{e^{-\gamma}-1} \right ),
\end{align}$\textbf{u}\in[0,1]^2$, where dependence parameter $\gamma \in (-\infty,\infty) \setminus \{0\}$. Independence occurs when $\gamma = 0$, and greater values of $\gamma$ indicate greater concordance. Similar to the Gaussian copula, Frank copula has no tail-dependence, but despite the normal margins leads to a jointly non-normal distribution.
\end{itemize}

For Gumbel, Clayton, Frank and Gaussian copulas, we consider the following respective parameter values: $\alpha = 1, 1.1, 1.2, 1.3, 1.4, 1.5,1.6$, $\theta = 0,0.2,0.4,0.6,0.8,1$, $\gamma = 0,0.5,1,1.5,2,2.5,3$ and $\rho = 0, 0.2, 0.4, 0.6, 0.8$, where the smallest value in each grid corresponds to independence. For distributions constructed using the $t$-copula, we first fix the correlation at zero and vary the tail-dependence parameter on $k = 0.1,0.2,0.3,...,1.9$. For any $k$, this leads to a jointly non-normal distribution with standard normal marginals, zero correlation and varying degrees of symmetric tail dependence. We then fix the degrees of freedom at $k=1$ and vary correlation on $\rho = 0,0.1,0.2,..0.7$. In both cases linear and tail dependence coexist, and no values on the parameter grid correspond to independence. 

For each of the five copulas, and each of the points on the parameter grids, we draw $S = 1,000$ random samples of size $n=50$, and for every draw, obtain the critical values using $N = 500$ permutations following algorithm outlined in Section \ref{permutation_critical_values}. The proportion of rejections of the null of independence by each of the statistics is the corresponding test power and is shown in Figure \ref{permutation_power_curves}.

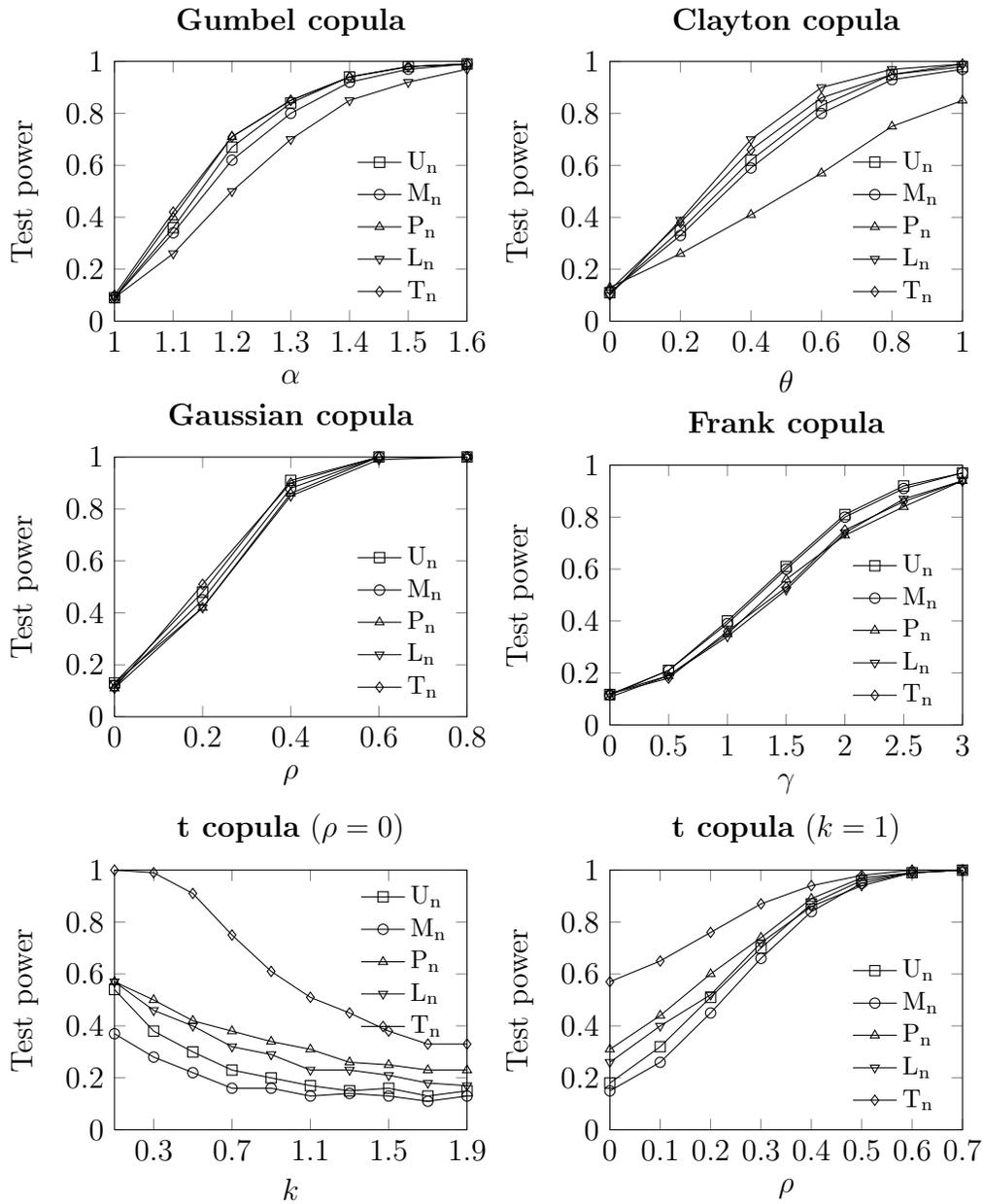
\begin{figure}
\input{PermutationPowerCurves.tex}
\caption{Proportion of rejections of the null of independence by statistics $U_n$, $M_n$, $P_n$, $L_n$ and $T_n$ at $10\%$ level of significance in bi-variate samples of size $n=50$ with standard normal marginals and Gumbel, Clayton, Frank, Gaussian ant $t$ copulas.} 
\label{permutation_power_curves}
\end{figure}

\subsection{The effect of weights on test power}

A number of conclusions are apparent from simulation results in Figure \ref{permutation_power_curves}:
\begin{itemize}
\item Firstly, the choice of the weights clearly has a significant effect on test power; in some cases, the difference in power between the least and most powerful weighted statistics is four-fold. For each of the power curves in Figure \ref{permutation_power_curves}, Table \ref{power_differences} shows the best and the worst-performing statistic (with no ranking when curves cross), and the maximum absolute power difference between the most and least-powerful tests.
\item Secondly, the effect of the weights on test power is greatest in the presence of tail dependence; for distributions constructed using Gaussian and Frank copulas, the effect of weights appears small.
\item Thirdly, the gains in test power appear to be greatest when the weights are set proportional to the deviations of the actual copula from independence.
\item Lastly, the symmetric tail weights appear to be a good choice for a versatile statistic; practically in all scenarios the power of $T_n$ was close to that of the optimal test.
\end{itemize}

\begin{table}[htbp]
\begin{center}
\begin{tabular}{|p{1.7cm}|p{2.3cm}|p{3cm}|p{3cm}|c|}
\hline
\textbf{Copula} & \textbf{Tail dependence} & \textbf{Greatest power} & \textbf{Lowest power} & \multicolumn{1}{l|}{\textbf{Difference}} \\ \hline
Gumbel & Upper & $P_n$ & $L_n$ & 0.21 \\ \hline
Clayton & Lower & $L_n$ & $P_n$ & 0.33 \\ \hline
Gauss & No & $U_n$ & - & 0.09 \\ \hline
Frank & No & $U_n$ & - & 0.08 \\ \hline
t ($\rho = 0$)& Symmetric & $T_n$ & $M_n$ & 0.72 \\ \hline
t (k = 1) & Symmetric & $T_n$ & $M_n$ & 0.41 \\ \hline
\end{tabular}
\end{center}
\caption{Locally most- and least-powerful statistics, and the greatest power difference, for the six dependence alternatives presented in Figure \ref{permutation_power_curves}.}
\label{power_differences}
\end{table}

\section{Discussion}
\label{section::discussion}

The extent to which the weights influence test power is surprising, particularly in the case of tail dependence, which warrants several extensions. In practice, the nature of dependence is rarely known beforehand. When the choice of the weights is not driven by applications, data-driven weights may provide a versatile solution to weight selection problem. Since weights that favor differences between $C$ and $C^\perp$ seem to yield greatest power advantage, a good candidate for $w(\textbf{u})$ could be a parametric estimate $\hat{w}(\textbf{u})$ which minimizes empirical differences $[C_n(\textbf{u}) - C^\perp(\textbf{u})]$, $\textbf{u}\in[0,1]^d$. The issue of data-driven weights, as well as the existence or non-existence of optimality conditions for classes of weighting functions under sequences of local alternatives both need to be explored.

Additionally, the issue of the weights needs to be considered in other contexts where independence tests are applied. For example, \cite{Quessy2010} shows how an independence statistic can be used to test for goodness of fit of Archimedian copulas. With relatively minor refinements, statistics proposed here can yield similar weighted tests for copula goodness of fit, in which case the weights may be used to adjust the sensitivity of the goodness-of-fit procedure to misfits of different parts of the copula (e.g. tails). This may be of considerable interest to practitioners working with copula models, for example, in financial risk modelling. 

\bibliographystyle{elsarticle-harv}
\bibliography{weightedtest}

\end{document}

%% file: PermutationPowerCurves.tex
% This file was created by matlab2tikz v0.4.2.
% Copyright (c) 2008--2013, Nico Schlömer <nico.schloemer@gmail.com>
% All rights reserved.
% 
% The latest updates can be retrieved from
%   http://www.mathworks.com/matlabcentral/fileexchange/22022-matlab2tikz
% where you can also make suggestions and rate matlab2tikz.
% 
% 
% 
\begin{tikzpicture}

\begin{axis}[%
width=1.9in,
height=1.4in,
scale only axis,
xmin=1,
xmax=5,
xtick={1,2,3,4,5},
xticklabels={0,0.2,0.4,0.6,0.8},
xlabel={$\rho$},
ymin=0,
ymax=1,
ylabel={Test power},
name=plot3,
title={\textbf{Gaussian copula}},
legend style={at={(0.97,0.03)},anchor=south east,fill=none,draw=none,legend cell align=left}
]
\addplot [
color=black,
solid,
mark=square,
mark options={solid}
]
table[row sep=crcr]{
1 0.13\\
2 0.48\\
3 0.91\\
4 1\\
5 1\\
};
\addlegendentry{\footnotesize $\text{U}_\text{n}$};

\addplot [
color=black,
solid,
mark=o,
mark options={solid}
]
table[row sep=crcr]{
1 0.12\\
2 0.45\\
3 0.88\\
4 1\\
5 1\\
};
\addlegendentry{\footnotesize $\text{M}_\text{n}$};

\addplot [
color=black,
solid,
mark=triangle,
mark options={solid}
]
table[row sep=crcr]{
1 0.11\\
2 0.42\\
3 0.86\\
4 1\\
5 1\\
};
\addlegendentry{\footnotesize $\text{P}_\text{n}$};

\addplot [
color=black,
solid,
mark=triangle,
mark options={solid,,rotate=180}
]
table[row sep=crcr]{
1 0.13\\
2 0.42\\
3 0.85\\
4 0.99\\
5 1\\
};
\addlegendentry{\footnotesize $\text{L}_\text{n}$};

\addplot [
color=black,
solid,
mark=diamond,
mark options={solid}
]
table[row sep=crcr]{
1 0.11\\
2 0.51\\
3 0.9\\
4 1\\
5 1\\
};
\addlegendentry{\footnotesize $\text{T}_\text{n}$};

\end{axis}

\begin{axis}[%
width=1.9in,
height=1.4in,
scale only axis,
xmin=1,
xmax=7,
xtick={1,2,3,4,5,6,7},
xticklabels={1,1.1,1.2,1.3,1.4,1.5,1.6},
xlabel={$\alpha$},
ymin=0,
ymax=1,
ylabel={Test power},
name=plot1,
at=(plot3.above north west),
anchor=below south west,
title={\textbf{Gumbel copula}},
legend style={at={(0.97,0.03)},anchor=south east,fill=none,draw=none,legend cell align=left}
]
\addplot [
color=black,
solid,
mark=square,
mark options={solid}
]
table[row sep=crcr]{
1 0.09\\
2 0.36\\
3 0.67\\
4 0.84\\
5 0.94\\
6 0.98\\
7 0.99\\
};
\addlegendentry{\footnotesize $\text{U}_\text{n}$};

\addplot [
color=black,
solid,
mark=o,
mark options={solid}
]
table[row sep=crcr]{
1 0.09\\
2 0.34\\
3 0.62\\
4 0.8\\
5 0.92\\
6 0.97\\
7 0.99\\
};
\addlegendentry{\footnotesize $\text{M}_\text{n}$};

\addplot [
color=black,
solid,
mark=triangle,
mark options={solid}
]
table[row sep=crcr]{
1 0.09\\
2 0.4\\
3 0.71\\
4 0.85\\
5 0.94\\
6 0.98\\
7 0.99\\
};
\addlegendentry{\footnotesize $\text{P}_\text{n}$};

\addplot [
color=black,
solid,
mark=triangle,
mark options={solid,,rotate=180}
]
table[row sep=crcr]{
1 0.09\\
2 0.26\\
3 0.5\\
4 0.7\\
5 0.85\\
6 0.92\\
7 0.97\\
};
\addlegendentry{\footnotesize $\text{L}_\text{n}$};

\addplot [
color=black,
solid,
mark=diamond,
mark options={solid}
]
table[row sep=crcr]{
1 0.1\\
2 0.42\\
3 0.71\\
4 0.85\\
5 0.94\\
6 0.98\\
7 0.99\\
};
\addlegendentry{\footnotesize $\text{T}_\text{n}$};

\end{axis}

\begin{axis}[%
width=1.9in,
height=1.4in,
scale only axis,
xmin=1,
xmax=6,
xtick={1,2,3,4,5,6},
xticklabels={0,0.2,0.4,0.6,0.8,1},
xlabel={$\theta$},
ymin=0,
ymax=1,
ylabel={Test power},
name=plot2,
at=(plot1.right of south east),
anchor=left of south west,
title={\textbf{Clayton copula}},
legend style={at={(0.97,0.03)},anchor=south east,fill=none,draw=none,legend cell align=left}
]
\addplot [
color=black,
solid,
mark=square,
mark options={solid}
]
table[row sep=crcr]{
1 0.11\\
2 0.35\\
3 0.62\\
4 0.83\\
5 0.95\\
6 0.98\\
};
\addlegendentry{\footnotesize $\text{U}_\text{n}$};

\addplot [
color=black,
solid,
mark=o,
mark options={solid}
]
table[row sep=crcr]{
1 0.11\\
2 0.33\\
3 0.59\\
4 0.8\\
5 0.93\\
6 0.97\\
};
\addlegendentry{\footnotesize $\text{M}_\text{n}$};

\addplot [
color=black,
solid,
mark=triangle,
mark options={solid}
]
table[row sep=crcr]{
1 0.13\\
2 0.26\\
3 0.41\\
4 0.57\\
5 0.75\\
6 0.85\\
};
\addlegendentry{\footnotesize $\text{P}_\text{n}$};

\addplot [
color=black,
solid,
mark=triangle,
mark options={solid,,rotate=180}
]
table[row sep=crcr]{
1 0.1\\
2 0.39\\
3 0.7\\
4 0.9\\
5 0.97\\
6 0.99\\
};
\addlegendentry{\footnotesize $\text{L}_\text{n}$};

\addplot [
color=black,
solid,
mark=diamond,
mark options={solid}
]
table[row sep=crcr]{
1 0.12\\
2 0.38\\
3 0.66\\
4 0.86\\
5 0.95\\
6 0.99\\
};
\addlegendentry{\footnotesize $\text{T}_\text{n}$};

\end{axis}

\begin{axis}[%
width=1.9in,
height=1.4in,
scale only axis,
xmin=1,
xmax=7,
xtick={1,2,3,4,5,6,7},
xticklabels={0,0.5,1,1.5,2,2.5,3},
xlabel={$\gamma$},
ymin=0,
ymax=1,
ylabel={Test power},
name=plot4,
at=(plot2.below south west),
anchor=above north west,
title={\textbf{Frank copula}},
legend style={at={(0.97,0.03)},anchor=south east,fill=none,draw=none,legend cell align=left}
]
\addplot [
color=black,
solid,
mark=square,
mark options={solid}
]
table[row sep=crcr]{
1 0.117\\
2 0.21\\
3 0.4\\
4 0.61\\
5 0.81\\
6 0.92\\
7 0.97\\
};
\addlegendentry{\footnotesize $\text{U}_\text{n}$};

\addplot [
color=black,
solid,
mark=o,
mark options={solid}
]
table[row sep=crcr]{
1 0.115\\
2 0.21\\
3 0.39\\
4 0.6\\
5 0.8\\
6 0.91\\
7 0.97\\
};
\addlegendentry{\footnotesize $\text{M}_\text{n}$};

\addplot [
color=black,
solid,
mark=triangle,
mark options={solid}
]
table[row sep=crcr]{
1 0.107\\
2 0.19\\
3 0.35\\
4 0.56\\
5 0.73\\
6 0.84\\
7 0.94\\
};
\addlegendentry{\footnotesize $\text{P}_\text{n}$};

\addplot [
color=black,
solid,
mark=triangle,
mark options={solid,,rotate=180}
]
table[row sep=crcr]{
1 0.121\\
2 0.19\\
3 0.34\\
4 0.52\\
5 0.74\\
6 0.87\\
7 0.94\\
};
\addlegendentry{\footnotesize $\text{L}_\text{n}$};

\addplot [
color=black,
solid,
mark=diamond,
mark options={solid}
]
table[row sep=crcr]{
1 0.12\\
2 0.18\\
3 0.36\\
4 0.53\\
5 0.75\\
6 0.86\\
7 0.94\\
};
\addlegendentry{\footnotesize $\text{T}_\text{n}$};

\end{axis}

\begin{axis}[%
width=1.9in,
height=1.4in,
scale only axis,
xmin=1,
xmax=8,
xtick={1,2,3,4,5,6,7,8},
xticklabels={0,0.1,0.2,0.3,0.4,0.5,0.6,0.7},
xlabel={$\rho$},
ymin=0,
ymax=1,
ylabel={Test power},
name=plot6,
at=(plot4.below south west),
anchor=above north west,
title={\textbf{t copula} ($k = 1$)},
legend style={at={(0.97,0.03)},anchor=south east,fill=none,draw=none,legend cell align=left}
]
\addplot [
color=black,
solid,
mark=square,
mark options={solid}
]
table[row sep=crcr]{
1 0.18\\
2 0.32\\
3 0.51\\
4 0.7\\
5 0.87\\
6 0.96\\
7 0.99\\
8 1\\
};
\addlegendentry{\footnotesize $\text{U}_\text{n}$};

\addplot [
color=black,
solid,
mark=o,
mark options={solid}
]
table[row sep=crcr]{
1 0.15\\
2 0.26\\
3 0.45\\
4 0.66\\
5 0.84\\
6 0.95\\
7 0.99\\
8 1\\
};
\addlegendentry{\footnotesize $\text{M}_\text{n}$};

\addplot [
color=black,
solid,
mark=triangle,
mark options={solid}
]
table[row sep=crcr]{
1 0.31\\
2 0.44\\
3 0.6\\
4 0.74\\
5 0.89\\
6 0.97\\
7 0.99\\
8 1\\
};
\addlegendentry{\footnotesize $\text{P}_\text{n}$};

\addplot [
color=black,
solid,
mark=triangle,
mark options={solid,,rotate=180}
]
table[row sep=crcr]{
1 0.26\\
2 0.4\\
3 0.52\\
4 0.72\\
5 0.86\\
6 0.94\\
7 0.99\\
8 1\\
};
\addlegendentry{\footnotesize $\text{L}_\text{n}$};

\addplot [
color=black,
solid,
mark=diamond,
mark options={solid}
]
table[row sep=crcr]{
1 0.57\\
2 0.65\\
3 0.76\\
4 0.87\\
5 0.94\\
6 0.98\\
7 1\\
8 1\\
};
\addlegendentry{\footnotesize $\text{T}_\text{n}$};

\end{axis}

\begin{axis}[%
width=1.9in,
height=1.4in,
scale only axis,
xmin=1,
xmax=10,
xtick={2,4,6,8,10},
xticklabels={0.3,0.7,1.1,1.5,1.9},
xlabel={$k$},
ymin=0,
ymax=1,
ylabel={Test power},
at=(plot6.left of south west),
anchor=right of south east,
title={\textbf{t copula} ($\rho = 0$)},
legend style={fill=none,draw=none,legend cell align=left}
]
\addplot [
color=black,
solid,
mark=square,
mark options={solid}
]
table[row sep=crcr]{
1 0.54\\
2 0.38\\
3 0.3\\
4 0.23\\
5 0.2\\
6 0.17\\
7 0.15\\
8 0.16\\
9 0.13\\
10 0.15\\
};
\addlegendentry{\footnotesize $\text{U}_\text{n}$};

\addplot [
color=black,
solid,
mark=o,
mark options={solid}
]
table[row sep=crcr]{
1 0.37\\
2 0.28\\
3 0.22\\
4 0.16\\
5 0.16\\
6 0.13\\
7 0.14\\
8 0.13\\
9 0.11\\
10 0.13\\
};
\addlegendentry{\footnotesize $\text{M}_\text{n}$};

\addplot [
color=black,
solid,
mark=triangle,
mark options={solid}
]
table[row sep=crcr]{
1 0.57\\
2 0.5\\
3 0.42\\
4 0.38\\
5 0.34\\
6 0.31\\
7 0.26\\
8 0.25\\
9 0.23\\
10 0.23\\
};
\addlegendentry{\footnotesize $\text{P}_\text{n}$};

\addplot [
color=black,
solid,
mark=triangle,
mark options={solid,,rotate=180}
]
table[row sep=crcr]{
1 0.57\\
2 0.46\\
3 0.4\\
4 0.32\\
5 0.29\\
6 0.23\\
7 0.23\\
8 0.21\\
9 0.18\\
10 0.17\\
};
\addlegendentry{\footnotesize $\text{L}_\text{n}$};

\addplot [
color=black,
solid,
mark=diamond,
mark options={solid}
]
table[row sep=crcr]{
1 1\\
2 0.99\\
3 0.91\\
4 0.75\\
5 0.61\\
6 0.51\\
7 0.45\\
8 0.38\\
9 0.33\\
10 0.33\\
};
\addlegendentry{\footnotesize $\text{T}_\text{n}$};

\end{axis}
\end{tikzpicture}%